\documentclass[psfig]{article}
\usepackage{amssymb, amsmath, amsthm}
\usepackage{MnSymbol}
\usepackage{esint}
\usepackage{wasysym}
\usepackage{graphicx}
\usepackage{color}
\usepackage{cite}
\usepackage[automark]{scrpage2}
\theoremstyle{plain}
\newtheorem{thm}{Theorem}

\newtheorem{cor}[thm]{Corollary}
\newtheorem{lem}[thm]{Lemma}

\theoremstyle{definition}

\theoremstyle{remark}
\newtheorem{rem}[thm]{Remark}

\newcommand{\norm}[1]{\left\Vert#1\right\Vert}

\newcommand{\Real}{\mathbb R}

\begin{document}
\bibliographystyle{plain}

\title{Vanishing viscosity as a selection principle for the Euler equations: The case of 3D shear flow}
\author{Claude Bardos\footnote{Laboratoire J.-L. Lions, Universit\'{e} de Paris VII ``Denis Diderot'', 75009 Paris, France} ,  Edriss S. Titi\footnote{Department of Mathematics and Department of Mechanical and Aero-space Engineering,
University of California, Irvine, CA 92697-3875, USA. \emph{Also}: The Department of Computer Science and Applied Mathematics,
The Weizmann Institute of Science, Rehovot 76100, Israel.} , and  Emil Wiedemann\footnote{Mathematisches Institut, Universit\"{a}t Leipzig, 04103 Leipzig, Germany}}
\date{August 11, 2012}
\maketitle
\begin{abstract}
We show that for a certain family of initial data, there exist non-unique weak solutions to the 3D incompressible Euler equations satisfying the weak energy inequality, whereas the weak limit of every sequence of Leray-Hopf weak solutions for the Navier-Stokes equations, with the same initial data, and as the viscosity tends to zero, is uniquely determined and equals the shear flow solution of the Euler equations. This simple example suggests that, also in more general situations, the vanishing viscosity limit of the Navier-Stokes equations could serve as a uniqueness criterion for weak solutions of the Euler equations.
\end{abstract}

\section{Introduction}
Consider the incompressible Euler equations,
\begin{equation}\label{euler}
\begin{aligned}
\partial_tv+v\cdot\nabla v+\nabla p&=0\\
\operatorname{div}v&=0,
\end{aligned}
\end{equation}
on the $d$-dimensional Torus $\mathbb{T}^d=(-1/2,1/2)^d$, $d\geq2$. A weak solution (i.e. a solution in the sense of distributions) $v\in C([0,T];L^2_w(\mathbb{T}^d))$ is said to be \emph{admissible} if it satisfies the weak energy inequality, i.e. if
\begin{equation}\label{energy}
\frac{1}{2}\int_{\mathbb{T}^d}|v(x,t)|^2dx\leq\frac{1}{2}\int_{\mathbb{T}^d}|v(x,0)|^2dx
\end{equation}
for every $t\in[0,T]$. Here, $T\in(0,\infty]$ and $C([0,T];L^2_w(\mathbb{T}^d))$ denotes the space of vector fields that are continuous as maps from $[0,T]$ into $L^2(\mathbb{T}^d)$ with respect to the weak topology in $L^2$.

It has been established in \cite{euler2} and \cite{euleryoung} (see Corollary 3 therein) that there exists a large set of initial data which admit infinitely many admissible weak solutions of (\ref{euler}). Notice that the cited results can be adapted to the case of periodic boundary conditions, considered here, in a straightforward manner. These initial data, however, are constructed in a rather abstract way. Recently, L. Sz\'{e}kelyhidi Jr. \cite{vortexpaper} exhibited the first example of concretely given initial data with the above mentioned  non-uniqueness property. More precisely, consider the flat vortex sheet in two dimensions given by
\begin{equation}\label{vortexsheet}
v_0(x)=\begin{cases}
e_1 & \text{if $x_2>0$}\\
-e_1 & \text{if $x_2<0$.}
\end{cases}
\end{equation}
Obviously, the stationary solution $v(\cdot,t)=v_0$ for all $t>0$ satisfies the Cauchy problem for the Euler equations in the weak sense. Sz\'{e}kelyhidi's result can be stated as follows:
\begin{thm}\label{vortexthm}
There exist $T>0$ and infinitely many weak solutions to the Euler equations in $\mathbb{T}^2\times[0,T]$ with initial data $v_0$ and pressure zero. Among these, infinitely many conserve the kinetic energy in time, and infinitely many have strictly decreasing energy.
\end{thm}
Here, the kinetic energy is defined as in equation (\ref{energy}) above. Sz\'{e}kelyhidi also observed that, clearly, any sequence of the Leray-Hopf weak solutions of the Navier-Stokes equations, corresponding to the initial data (\ref{vortexsheet}), converges weak$-*$ (in fact even strongly) in $L^{\infty}([0,T];L^2(\mathbb{T}^2))$ to the stationary solution of the Euler equations, as the viscosity tends to zero. Hence, being a vanishing viscosity limit distinguishes the stationary solution from all the other weak solutions of the Euler equations stated in Theorem \ref{vortexthm}. The aim of the present note is to prove a similar statement for the slightly more sophisticated case of the three-dimensional shear flow.

\section{3D shear flow}
In three dimensions, any initial data of the form $v_0(x)=(v_1(x_2),0,v_3(x_1,x_2))$ has the shear flow solution of the Euler equations given by
\begin{equation}\label{shearflow}
v(x,t)=(v_1(x_2),0,v_3(x_1-tv_1(x_2),x_2))
\end{equation}
(see \cite{Bardos-Titi} and \cite{DM} for results concerning this shear flow). Note in particular that $v(x,t)$ is periodic in $x$ if $v_0$ is.

Sz\'{e}kelyhidi's result can now easily be extended to the case of three-dimensional shear flows as follows:
\begin{cor}
Let $v_0(x)=(v_1(x_2),0,v_3(x_1,x_2))$, where
\begin{equation*}
v_1(x_2)=\begin{cases}
1 & \text{if $ ~0<x_2 < 1/2$}\\
-1 & \text{if $-1/2<x_2<0$} ,
\end{cases}
\end{equation*}
extended periodically, with basic period $(-1/2,1/2)$,
and $v_3\in L^2_{x_1,x_2}(\mathbb{T}^2)$ is arbitrary. Then there exist $T>0$, and infinitely many admissible weak solutions of the 3D Euler equations on $\mathbb{T}^3\times[0,T]$ with initial data $v_0$.
\end{cor}
\begin{proof}
Take $u(x,t)=(u_1(x_1,x_2,t),u_2(x_1,x_2,t))$ to be a solution to the 2D vortex sheet problem as in Theorem \ref{vortexthm}. Then, the triple
\begin{equation*}
(u_1(x_1,x_2,t),u_2(x_1,x_2,t),w(x_1,x_2,t))
\end{equation*}
will be a weak solution of the 3D Euler equations (with zero pressure and initial data $v_0$) if $w$ is a weak solution of the 2D transport equation
\begin{equation}\label{transport}
\begin{aligned}
\partial_tw+u\cdot\nabla w&=0\\
w(t=0)&=v_3.
\end{aligned}
\end{equation}
Such a solution $w\in L^{\infty}((0,T);L^2(\mathbb{T}^2))$ exists; see Proposition II.1 in \cite{dipernalions}, which clearly holds also in the periodic setting. Moreover, we may even assume $w\in C((0,T);L^2_w(\mathbb{T}^2))$, see Appendix A of \cite{euler2}. Finally, as one can see from the proof of the cited Proposition II.1, we have
\begin{equation*}
\norm{w(\cdot,t)}_{L^2(\mathbb{T}^2)}\leq\norm{v_3}_{L^2(\mathbb{T}^2)}\hspace{0.5cm}\text{for every $t>0$}
\end{equation*}
(this is due to the weak lower semi-continuity of the norm in $L^{\infty}((0,T);L^2(\mathbb{T}^2))$). Hence our solution $(u_1,u_2,w)$ is an admissible weak solution of the 3D Euler equations.
\end{proof}
\begin{rem}
It is not possible to deduce from (\ref{transport}) that $\norm{w(\cdot,t)}_{L^2}$ is conserved in time: For this to hold, $w$ would have to be a renormalised solution in the sense of DiPerna-Lions \cite{dipernalions}, which is ruled out by the irregularity of Sz\'{e}kelyhidi's solutions.
\end{rem}
Before we state and prove the main result of this note, we need an auxiliary result:
\begin{lem}\label{uniquelemma}
Let $v\in L^2(\mathbb{T};\Real)$ and $w_0\in L^2(\mathbb{T}^2)$. Then the Cauchy problem for the linear transport equation
\begin{equation*}
\begin{aligned}
\partial_tw(x_1,x_2,t)+v(x_2)\partial_{x_1}w(x_1,x_2,t)&=0\\
w(\cdot,0)&=w_0
\end{aligned}
\end{equation*}
has a solution $w \in C([0,T];L_w^2(\mathbb{T}^2))$, satisfying the equation in the sense of distributions, and this solution is unique in the class $L^{\infty}((0,T);L^2(\mathbb{T}^2))$.
\end{lem}
We omit the elementary proof of the Lemma.

\begin{thm}
Let again $v_0(x)=(v_1(x_2),0,v_3(x_1,x_2))$, where we assume $v_1\in L^2(\mathbb{T})$ and $v_3\in L^2(\mathbb{T}^2)$. Then, for every viscosity $\nu>0$, there exists a unique Leray-Hopf weak solution of the Navier-Stokes equations with viscosity $\nu$ and initial data $v_0$, and these solutions $u^{\nu}$ converge weak$-*$ in $L^{\infty}([0,T];L^2(\mathbb{T}^3))$ to the shear flow (\ref{shearflow}) corresponding to $v_0$, as $\nu\rightarrow0$.
\end{thm}
\begin{proof}
Let $\nu>0$. The intuition that the solution of Navier-Stokes should preserve the particular structure of the initial data leads us to the ansatz
\begin{equation*}
u^{\nu}(x,t)=(u^{\nu}_1(x_2,t),0,u^{\nu}_3(x_1,x_2,t))
\end{equation*}
and $p^{\nu}=0$, where $p^{\nu}$ denotes the pressure. Inserting this into the Navier-Stokes equations gives the so called two-and-half Navier-Stokes equations
\begin{equation*}
\begin{aligned}
\partial_t u^{\nu}_1(x_2,t)-\nu\partial_{x_2}^2u^{\nu}_1(x_2,t)&=0\\
\partial_t u^{\nu}_3(x_1,x_2,t)+u^{\nu}_1(x_2,t)\partial_{x_1}u^{\nu}_3(x_1,x_2,t)-\nu\Delta_{x_1,x_2}u^{\nu}_3(x_1,x_2,t)&=0,
\end{aligned}
\end{equation*}
which is known to be globally well-posed for this kind of initial data (see, e.g., \cite{DM}, and see \cite{Constantin-Foias} for the global existence and uniqueness of weak solutions of the 2D Navier-Stokes equations). For completion, we observe that the first equation is simply the one-dimensional heat equation with initial data $v_1(x_2)$, whose solution obviously converges to the time-independent function $v_1(x_2)$ strongly in $L^2(\mathbb{T}\times[0,T])$, as the viscosity tends to zero. The second equation is an advection-diffusion equation. By standard parabolic theory (see e.g. Theorem 5 in Section 7.1.3 of \cite{evanspde}, which again can be adapted to the periodic case, or the 2D Navier-Stokes theory) there exists a (unique) solution
\begin{equation*}
u^{\nu}_3\in L^2([0,T];H^1(\mathbb{T}^2))\bigcap C([0,T];L^2(\mathbb{T}^2))
\end{equation*}
(recall that the initial data $v_3$ is in $L^2$). Hence, for every fixed $\nu > 0$,  we obtain a Leray-Hopf weak solution with the initial data  $v_0(x)$. Following ideas from \cite{Serrin}  (see, e.g., \cite{BLNNT} and \cite{Iftimie-Raugel} for details) one can show that this solution is unique within the class of all 3D Leray-Hopf weak solutions; moreover, this solution depends continuously on the initial data, when the initial data  is perturbed in the $L^2(\mathbb{T}^3)$ norm (see \cite{BLNNT}).
%
Furthermore, since the family of unique solutions, $u^{\nu}$, is uniformly bounded in $L^{\infty}_tL^2_x$, there exists a subsequence $u^{\nu_k}$ which converges weak$-*$ to $u\in L^{\infty}_tL^2_x$, and $u$ satisfies
\begin{equation}\label{limiteq}
\begin{aligned}
u_1(x_2,t)&=v_1(x_2)\\
u_2&=0\\
\partial_t u_3(x_1,x_2,t)+v_1(x_2)\partial_{x_1}u_3(x_1,x_2,t)&=0\\
u_3(x_1,x_2,0)&=v_3(x_1,x_2).
\end{aligned}
\end{equation}
Indeed, the equation for $u_3$ follows from $u_1^{\nu}u_3^{\nu}\stackrel{*}{\rightharpoonup}u_1u_3$, thanks to  the strong convergence of $u_1^{\nu}$ to $u_1$. Next, it follows from Lemma \ref{uniquelemma} above that system (\ref{limiteq}) has a \emph{unique} solution, and that this unique solution is given precisely by the shear flow (\ref{shearflow}) (see \cite{Bardos-Titi}). Finally, this uniqueness implies that the whole sequence $u^{\nu}$, and not just a subsequence, converges to the shear flow solution (\ref{shearflow}).
\end{proof}
Combining the Corollary and the Proposition, we see that among the infinitely many admissible solutions of the Euler equations that are corresponding to the shear flow initial data, $v_0(x) = (v_1(x_2),0,v_3(x_1))$, the shear flow solution given by (\ref{shearflow}) has the exclusive property of being a vanishing viscosity limit. \\

\textbf{Acknowledgement.} Part of this work was done while the third author was a visitor to the project ``Instabilities in Hydrodynamics'' of the Fondation Sciences Math\'{e}matiques de Paris. He gratefully acknowledges the Fondation's support. The work of the second author is supported in part by the NSF grants  DMS-1009950, DMS-1109640 and DMS-1109645, and by the Minerva Stiftung/Foundation.


\begin{thebibliography}{10}

\bibitem{BLNNT}
Claude Bardos, Milton Lopes~Filho, Dongjuan Niu, Helena Nussenzveig~Lopes, and
  Edriss~S. Titi.
\newblock {S}tability of viscous, and instability of non-viscous, 2d weak
  solutions of incompressible fluids under 3d perturbations.
\newblock {\em Preprint}, http://arxiv.org/pdf/1201.2742v2.

\bibitem{Bardos-Titi}
Claude Bardos and Edriss~S. Titi.
\newblock Loss of smoothness and energy conserving rough weak solutions for the
  {$3d$} {E}uler equations.
\newblock {\em Discrete Contin. Dyn. Syst. Ser. S}, 3(2):185--197, 2010.

\bibitem{Constantin-Foias}
Peter Constantin and Ciprian Foias.
\newblock {\em Navier-{S}tokes equations}.
\newblock Chicago Lectures in Mathematics. University of Chicago Press,
  Chicago, IL, 1988.

\bibitem{euler2}
Camillo De~Lellis and L{\'a}szl{\'o} Sz{\'e}kelyhidi, Jr.
\newblock On admissibility criteria for weak solutions of the {E}uler
  equations.
\newblock {\em Arch. Ration. Mech. Anal.}, 195(1):225--260, 2010.

\bibitem{dipernalions}
R.~J. DiPerna and P.-L. Lions.
\newblock Ordinary differential equations, transport theory and {S}obolev
  spaces.
\newblock {\em Invent. Math.}, 98(3):511--547, 1989.

\bibitem{DM}
Ronald~J. DiPerna and Andrew~J. Majda.
\newblock Oscillations and concentrations in weak solutions of the
  incompressible fluid equations.
\newblock {\em Comm. Math. Phys.}, 108(4):667--689, 1987.

\bibitem{evanspde}
Lawrence~C. Evans.
\newblock {\em Partial differential equations}, volume~19 of {\em Graduate
  Studies in Mathematics}.
\newblock American Mathematical Society, Providence, RI, second edition, 2010.

\bibitem{Iftimie-Raugel}
Drago{\c{s}} Iftimie and Genevi{\`e}ve Raugel.
\newblock Some results on the {N}avier-{S}tokes equations in thin 3{D} domains.
\newblock {\em J. Differential Equations}, 169(2):281--331, 2001.
\newblock Special issue in celebration of Jack K. Hale's 70th birthday, Part 4
  (Atlanta, GA/Lisbon, 1998).

\bibitem{Serrin}
James Serrin.
\newblock The initial value problem for the {N}avier-{S}tokes equations.
\newblock In {\em Nonlinear {P}roblems ({P}roc. {S}ympos., {M}adison, {W}is.,
  1962)}, pages 69--98. Univ. of Wisconsin Press, Madison, Wis., 1963.

\bibitem{vortexpaper}
L{\'a}szl{\'o} Sz{\'e}kelyhidi.
\newblock Weak solutions to the incompressible {E}uler equations with vortex
  sheet initial data.
\newblock {\em C. R. Math. Acad. Sci. Paris}, 349(19-20):1063--1066, 2011.

\bibitem{euleryoung}
L{\'a}szl{\'o} Sz{\'e}kelyhidi, Jr. and Emil Wiedemann.
\newblock {Y}oung measures generated by ideal incompressible fluid flows.
\newblock {\em To appear in Arch. Ration. Mech. Anal.}

\end{thebibliography}

\end{document}